\documentclass[12pt,a4paper]{article}

\usepackage[latin1]{inputenc}

\usepackage{amsmath, relsize, enumerate}
\usepackage{amsfonts, amsthm, dsfont}
\usepackage{amssymb}
\usepackage{graphicx}
\usepackage{bm}

\def\g{{\gamma}}

\def\e{{\varepsilon}}

\def\k{{\kappa}}
\def\l{{\lambda}}

\def\1{{\mathds{1}}}

\def\E{{\mathds{E}}}

\def\N{{\mathds{N}}}
\def\P{{\mathds{P}}}

\def\R{{\mathds{R}}}
\def\S{{\mathds{S}}}
\def\Z{{\mathds{Z}}}

\def\cF{\mathcal{F}}
\def\cG{\mathcal{G}}
\def\cK{\mathcal{K}}
\def\cL{\mathcal{L}}
\def\cP{\mathcal{P}}

\def\bd{\partial}

\newtheorem{theorem}{Theorem}[section]
\newtheorem{corollary}[theorem]{Corollary}
\newtheorem{lemma}[theorem]{Lemma}
\theoremstyle{definition}

\theoremstyle{remark}

\begin{document}
\title{Expected mean width of the randomized integer convex hull}

\author{Hong Ngoc Binh, Matthias Reitzner}


\date{}

\maketitle

\begin{abstract}
Let $K \in \R^d$ be a convex body, and assume that $L$ is a randomly rotated and shifted integer lattice. Let $K_L$ be the convex hull of the (random) points $K \cap L$.
The mean width $W(K_L)$ of $K_L$ is investigated. The asymptotic order of the mean width difference $W(\l K)-W((\l K)_L)$ is maximized by the order obtained by polytopes and minimized by the order for smooth convex sets as $\l \to \infty$.
\end{abstract}



\section{Introduction}

Let $K$ be a convex body and $\Z^d$ the integer lattice in $\R^d$. The convex hull $[K \cap \Z^d]$ of the intersection of $K$ with $\Z^d$ yields a polytope $K_{\Z^d}$, the integer convex hull of $K$. The higher dimensional Gauss circle problem asks for $K=\l B^d$, the ball of radius $\l>0$, how many integer points are contained in $K_{\Z^d}$ compared to its volume $V_d(K)$. The metric variant we consider here compares the volume of $K_{\Z^d}$ to the volume of $K$, and more generally the intrinsic volumes $V_j(K_{\Z^d})$ to the intrinsic volumes of $K$. This problem has a long history, and more recent investigations have been motivated by questions from integer programming and enumeration problems. We refer to the article by B\'ar\'any and Larman \cite{BarLarman98} and the survey article by B\'ar\'any \cite{Bar2008} for more details.

It is immediate that all these problems depend on the position, size and shape of $K$ in a delicate way. Consider e.g. the enlarged unit cube $K=\l C^d=[-\l,\l]^d$, $\l>0$, where all functionals of $K_{\Z^d}$ are locally constant for $\l \notin \N$ and have jumps at $\l \in \N$. This is due to the fact that $C^d$ is in a special position with respect to $\Z^d$. Therefore it is of interest to ask what happens in \emph{generic} situations. 

This question was made precise by B\'ar\'any and Matou\v{s}ek \cite{BarMat2005} who investigated the integer convex hull of $\l K$ when $K$ is in a random position, i.e. $K$ is a randomly rotated and shifted copy of a convex body $K_0$. Alternatively, one can intersect $K_0$ with a random lattice $L$, a randomly shifted and rotated copy of $\Z^d$, which yields the randomized integer convex hull, $ K_L = [ K \cap L] $.Of interest are metric quantities of this random polytope like the volume, surface area, mean width, and combinatorial quantities like the number of faces.

\medskip
This problem turns out to be surprisingly difficult even in simple cases. B\'ar\'any and Matou\v{s}ek proved that the expected number of vertices of $K_L$ is connected to the so-called floating body of $K$, if the boundary of $K$ is sufficiently smooth. Further, in the planar case, they proved integral bounds for the expected area difference $V_2(\l K)-V_2((\l K)_L)$ which led to the bounds
\begin{equation}\label{eq:BarMat}
c_1 \underbrace{\left( \ln V_2(\l K) \right)^2}_{\approx (\ln \l)^2} 
\leq 
V_2(\l K)- \E V_2((\l K)_L) 
\leq
c_2 \underbrace{V_2(\l K)^{\frac 13}}_{\approx \l^{\frac 23}} 
\end{equation}
for $\l $ sufficiently large. The lower bound is attained for polygons, and the upper bound for smooth convex sets.
We are not aware of any other results on the random integer convex hull $K_L$.

\medskip
Surprisingly the behaviour in formula \eqref{eq:BarMat} changes if we consider the mean width instead of the area. To define the mean width, consider for given $u \in \S^{d-1}$ two parallel hyperplanes orthogonal to $u$ squeezing $K$. The distance between these two hyperplanes is the width $W(K,u)$ in direction $u$ (for a formal definition see Section \ref{sec:notations}). The mean width is given by 
$$W(K)=\frac{1}{\omega_d}\int\limits_{\S ^{d-1}}W(K,u)du . $$
Up to a constant, the mean width is in the planar case the perimeter $P(K)$ of a convex body $K$, and in general dimensions the first intrinsic volume (for the definition of intrinsic volume we refer to Section \ref{sec:general}).

The first main result of our paper gives upper and lower bounds on the expected mean width difference.

\begin{theorem}\label{thm:genbounds}
Let $K$ be an arbitrary convex body. Then there are constants $ \g_1(K), \g_2 (K)$ such that
$$
\g_1 (K) \l^{-\frac{d-1}{d+1}}
\leq
W(\l K) - \E (W((\l K)_L))
\leq
\g_2 (K)
$$
as long as $\l \geq \l(K)$.
\end{theorem}
Hence in the planar case, \eqref{eq:BarMat} can be complemented by an inequality for the perimeter difference:
\begin{equation*}
 \g_1 (K) \underbrace{P(\l K)^{- \frac 13}}_{ \approx \l^{-\frac{1}{3}}}
\leq 
P(\l K)- \E P((\l K)_L) 
\leq
\g_2 (K) .
\end{equation*}

Note that the upper bound of Theorem \ref{thm:genbounds} can be generalized to all intrinsic volumes.
\begin{corollary}\label{cor:intrinsicvol}
Let $K$ be a convex body. Then there is a constants $\g_2(K)$, such that 
$$ V_j(\l K)-\E (V_j((\l K)_L)) \leq \g_2 (K) V_{j-1}(\l K)$$
for all $j \in \{1, \dots, d\}$, as long as $\l \geq \l(K)$.
\end{corollary}
Yet it is clear from \eqref{eq:BarMat} that this inequality is not optimal in general. The area difference is maximized by smooth convex sets, in contrast to the mean width difference which is maximized by polytopes. It would be of high interest to generalize the results of B\'ar\'any and Matou\v{s}ek \eqref{eq:BarMat} and Theorem \ref{thm:genbounds} to sharp inequalities for all intrinsic volumes.

\medskip
Theorem \ref{thm:genbounds} concerning the mean width is optimal as shown by polytopes and smooth convex sets.
\begin{theorem}\label{thm:pol}
Let $P$ be a polytope. Then there is a constant $ \g (P)>0$ such that
$$
\lim_{\l \to \infty} W(\l P) - \E (W((\l P)_L))
= \g (P)  .
$$
\end{theorem}
\begin{theorem}\label{thm:smooth}
Assume $K $ is a smooth convex body. Then there is a constant $\g_3 (K)$  such that
$$
W(\l K) - \E (W((\l K)_L)) \leq
\g_3 (K)  \l^{-\frac{d-1}{d+1}}
$$
for $\l$ sufficiently large.
\end{theorem}

Thus in the planar case and for smooth convex bodies, the mean width difference is of order $\l^{- \frac 13}$ which tends to zero, and by the result of B\'ar\'any and Matou\v{s}ek \cite{BarMat2005} the volume difference is of order $\l^{\frac 32}$ which tends to infinity.
To compare the two results heuristically, one should check that the volume difference is approximately the perimeter $P(\l K)$ times the mean distance of $\l K$ and $( \l K)_L$,  which is the mean width,
$$
V_2(\l K) - \E V_2 ((\l K)_L) \approx 
P(\l K) \left( W(\l K)- W((\l K)_L) \right)   \approx \l \, \l^{- \frac 13}  = \l^{\frac{2}{3}} 
$$
and thus these results fit together nicely. This simple observation breaks down for polytopes.

\bigskip
The paper is organized in the following way: Section \ref{sec:notations} and Section \ref{sec:basic} contain basic facts, in Section \ref{sec:general} we prove Theorem \ref{thm:genbounds} and Corollary \ref{cor:intrinsicvol}, Section \ref{sec:smooth} is devoted to smooth convex sets, and Section \ref{sec:polytopes} to investigations for polytopes.


\section{Notations}\label{sec:notations}

We work in $d$-dimensional Euclidean space $\R^d$ with inner product $\langle \cdot, \cdot \rangle$, and denote by $B^d$ its unit ball and by $\S^{d-1}= \bd B^d$ the unit sphere. Here $\bd K$ is the boundary of a set $K \subset \R^d$. 
By $B(x,r)$ we denote a ball with center $x$ and radius $r$.
The volume of $B^d$ is $\k_d$, and the spherical Lebesgue- or Hausdorff-measure of $\S^{d-1}$ is $\omega_d=d \k_d$.

Let $\cL$ be the set of rotated and translated integer lattices in $\R^d$, 
$$
\cL = 
\left\{L_{t,\rho}=\rho(\Z^d + t) \colon t\in [0,1)^d,\rho \in SO(d)  \right\} .
$$

For $A \subset \R^d$ we write $\int_{A} f(x) \, dx$ for integration with respect to the $d$-dimensional Lebesgue measure, 
and analogously $\int _A f(u)\, du$ for integration with respect to spherical Lebesgue measure for $A \subset \S^{d-1}$, 
$\int _A f(\rho)\, d\rho$ for integration with respect to the Haar probability measure for $A \subset SO_d$, 
and 
$$
\int\limits_{\cL} f(L) \, dL= \int\limits_{[0,1]^d} \int\limits_{SO_d} f(L_{t,\rho}) \, d\rho \, dt .
$$ 
Thus the \lq uniform measure\rq\  on $\cL$ is given by uniformly chosen $t\in [0,1]^d$ and $\rho \in SO_d$, and for a set $A \subset \cL$ we have 
$$\P( L \in A)= \int\limits_{\cL} \1 ( L \in A)\, dL  ,\ \  
\E f(L) = \int\limits_{\cL} f(L) \, dL \ ,$$
where $\1(\cdot)$ denotes the indicator function.

The convex hull of a set $A$ is denoted by $ [A]$, $\cK^d$ denotes the set of convex bodies, i.e. compact convex sets with nonempty interior, $\cP^d \subset \cK^d$ the set of convex polytopes.  For given $K \subset \cK^d$, the {\it randomized integer convex hull} is the random polytope defined by 
$$K_L:=[K\cap L],$$
where $L\in \mathcal{L}$ is chosen uniformly.

We are interested in the distance between $K$ and $K_L$. To define the distance let 
$$h_K(u)= \max \{ \langle x , u \rangle \colon x \in K \}$$
be the support function of $K \in \cK^d$ in direction $u \in \S^{d-1}$.
The Hausdorff distance $d_H(K,Q)$ between two convex bodies $K, Q$ is given by
\begin{equation}\label{def:dH}
d_H(K,Q) = \max_{u \in \S^{d-1}} |h_K(u)-h_Q(u)|.
\end{equation}

Note that 
$ H_K(u) = \{x \in \R^d \colon \langle x, u \rangle \geq h_K(u)  \}$
is a supporting halfspace to $K$ with unit normal vector $u$.
For each $u\in\S ^{d-1}$ and $t\in [0;+\infty)$, we denote by $K_{t,u}$ the cap of width $t$ cut off from $K$ by a halfspace parallel to $H_K(u)$,
$$
K_{t,u} = \{x \in K \colon \langle x, u \rangle \geq h_K(u)-t  \} .
$$

For $u\in \S ^{d-1}$, the width of $K$ in direction $u$ is defined by
$$W(K,u)=h_K(u)+h_K(-u) ,$$
and the mean width of $K$ is
$$W(K)=\frac{1}{\omega_d}\int\limits_{\S ^{d-1}}W(K,u)du=\frac{2}{\omega_d}\int\limits_{\S ^{d-1}}h_K(u)du . $$

\section{Basic results}\label{sec:basic}

We are interested in the distance between $K$ and $K_L$ measured in terms of the difference of the mean width $W(K)-W(K_L)$. Observe that since $K_L \subset K$ this difference is always postive and equals zero if and only if $K=K_L$. We start with a simple but crucial lemma.

\begin{lemma}\label{le:EW=intP}
Assume that $K \in \cK^d$. Then
$$ W(K)-\E (W(K_L))
=
\frac{2}{\omega_d}\int\limits_{\S ^{d-1}}\int\limits^{\infty}_{0} \P (K_{t,u}\cap L=\emptyset)dtdu
. $$
\end{lemma}
\begin{proof}
The difference of the expected mean width of $K$ and $K_L$ is by definition
\begin{align*}
W(K)-\E (W(K_L))
&=
\frac{2}{\omega_d} \ \E \int\limits_{\S ^{d-1}}(h_K(u)-h_{K_L}(u))\, du .
\end{align*}
Since $K_L \subset K $, the integrand is postive, and Fubinis theorem yields
\begin{align}\label{eq:Wdiff}
W(K)-\E (W(K_L))
&= \nonumber
\frac{2}{\omega_d} \ \E \int\limits_{\S ^{d-1}}\int\limits^{\infty}_{0}\1( t \leq h_K(u)-h_{K_L}(u))\, dtdu
\\ &= \nonumber
\frac{2}{\omega_d} \ \E \int\limits_{\S ^{d-1}}\int\limits^{\infty}_{0}\1(K_{t,u}\cap L=\emptyset) \, dtdu
\\ &= 
\frac{2}{\omega_d} \int\limits_{\S ^{d-1}}\int\limits^{\infty}_{0} \P (K_{t,u}\cap L=\emptyset)\, dtdu .
\end{align}
\end{proof}

Hence estimating the mean width difference boils down to estimate the probability that a cap avoids the random lattice $L$.
The following upper bound was stated by B\'ar\'any and Matou\v{s}ek \cite{BarMat2005} 
and proved by  B\'ar\'any \cite{Bar2007}.

\begin{lemma}\label{le:upper-bound-P}
There exist constants $\nu>0$ and $c>0$ (both depending on $d$) such that for every convex body $K \in \cK^d$ with  $V_d(K)\geq \nu $, 
$$ \P (K \cap L =\emptyset)\leq \frac{c}{V_d(K)} $$
holds.
\end{lemma}

We give a simple lower bound which turns out to have the right order in the applications we need in this work.
\begin{lemma}\label{le:lower-bound-P}  For any measurable set $A \subset \R^d$ we have
$$ \P (A \cap L =\emptyset)\geq 1- V_d(A) . $$
\end{lemma}

\begin{proof} 
We start by calculating the expected number of lattice points in $A$. 
\begin{align*}
\E (\#\{ A \cap L\}) 
=
\int\limits_{\cL}\sum\limits_{z\in L}\1_A(z) dL
&= 
\int\limits_{SO_d}\int\limits_{\ [0,1) ^{d}}\sum\limits_{\omega\in \Z^d}\1_A(\rho(\omega + t)) dtd\rho
\\ &= 
\int\limits_{SO_d}\sum\limits_{\omega\in \Z^d}\int\limits_{\ [0,1) ^{d}+\omega}\1_A(\rho(y)) dyd\rho
\\ &= 
\int\limits_{SO_d}\int\limits_{\R^{d}}\1_A(\rho(y)) dyd\rho
\\ &= 
\int\limits_{SO_d}\int\limits_{\R^{d}}\1_A(x) dxd\rho
= 
V_d(A)
\end{align*}
Therefore,
\begin{align*}
V_d(A)=\E (\# \{A\cap L\} )
&=
\sum^{\infty}_{i=1} i \P(\# \{A\cap L\} =i)
\\ &\geq
\sum^{\infty}_{i=1} \P(\# \{A\cap L\} =i)
\\ &=
\P (A\cap L \neq \emptyset),
\end{align*}
and hence, $\P(A\cap L =\emptyset)=1-\P(A\cap L \neq \emptyset)\geq 1-V_d(A)$.
\end{proof}

\section{General Convex Bodies}\label{sec:general}

In this section we prove bounds for general convex bodies.  We start with the upper bound.  
The following lemma is somehow connected to Khintchin's Flatness Theorem \cite{Khin1948}, see also \cite{KannLov1988}. It states that a cap which is too fat cannot avoid any lattice.

\begin{lemma}\label{le:L_meets_cap}
Let $K$ be a convex body. Then there are constants $ \tau (K), \l(K)$ such that for $t \geq \tau (K)$ and $\l \geq \l(K)$ we have
$$ (\l K)_{t,u} \cap L \neq \emptyset$$
for all $u \in \S^{d-1}$ and all $L \in \cL$.
\end{lemma}

\begin{proof}
Assume w.l.o.g. that the inball $E$ of $K$ is centered at the origin, and denote by $x(u) \in \bd K$ a boundary point with outer unit normal vector $u$. Then $K$ contains the cone
$$
C_u=[x(u), E \cap u^\perp]
$$
with base $ E \cap u^\perp$ and apex $x(u)$. Denote by $r(u)$ the radius of the inball of $C_u$. Thus $s r(u)$ is the inball of $s C_u$. Observe that any ball of radius at least $\frac {\sqrt d}2$ meets any lattice $L= \rho(\Z^d+t)$. Thus for
\begin{equation}\label{eq:largeball}
s = \frac{\sqrt d}{2 r(u) }
\end{equation}
the cone $s C_u$ must contain a lattice point.
The essential observation is that $s C_u$ is a cone with height  $s h_K(u)$, and that $ (\l C_u)_{t,u} $ also is a homothetic  copy of $C_u$ with height $t$.
Hence \eqref{eq:largeball} implies
$$
(\l C_u)_{t,u}  \cap L = 
\left( \frac t{h_K(u)} C_u \right)  \cap L 
\neq \emptyset
$$
for 
$$t \geq s h_K(u)= \frac{\sqrt d\, h_K(u)}{2 r(u)}$$
as long as $\l  \geq s=\frac {\sqrt d}{2 r(u)}$. We define
$$
\tau (K)  := \max_{u \in \S^{d-1}} \frac{\sqrt d\, h_K(u)}{2 r(u)}, \ 
\mbox{ and } \ 
\l (K)  := \max_{u \in \S^{d-1}} \frac{\sqrt d }{2 r(u)}
$$
and obtain for $t \geq \tau(K)$ and $\l \geq \l(K)$ 
$$
(\l C_u)_{t,u}  \cap L \neq \emptyset
$$
for all $u \in \S^{d-1}$ and $L \in \cL$. Since $ (\l C_u)_{t,u} \subset (\l K)_{t,u}$ this yields the lemma.

\end{proof}

There are some immediate consequences.
By Lemma \ref{le:L_meets_cap}, for each $u \in \S^{d-1}$ the distance of the support functions of $K$ and $K_L$ is at most $\tau(K)$, which by definition gives a simple upper bound for the Hausorff distance \eqref{def:dH}  and for the mean width difference \eqref{eq:Wdiff}. Putting $\g_2(K)=2 \tau(K)$ this is the stated upper bound in Theorem \ref{thm:genbounds}.

\begin{theorem}\label{th:genupperbound}
Let $K$ be a convex body. Then there is a constant $\tau(K)$ such that for $\l$ sufficiently large
\begin{equation}\label{eq:dH<tau}
 d_H(\l K, (\l K)_L) \leq \tau(K)  
\end{equation}
for any lattice $L \in \cL$, and
$$ W(\l K)-\E (W((\l K)_L)) \leq 2 \tau(K) . $$
\end{theorem}

\bigskip
The intrinsic volumes $V_j(K)$ of a convex body $K$, $j=0, \dots, d$, are defined as the coefficients in the Steiner formula,
$$
V_d(K + \e B^d) = \sum_{i=0}^d \k_i V_{d-i}(K) \e^i \ , 
$$
where e.g. $2V_{d-1}(K)$ is the surface area of $K$, $\frac{2 \kappa_{d-1}}{d \kappa_d} V_1(K)$ equals the mean width $W(K)$, and $V_0(K)=1$ is the Euler characteristic of $K$.
By Kubotas formula, the intrinsic volumes of a convex body can be written as
$$
V_{j} (K)
= c_{d,k,j}^{-1}  \int_{\cG_k^d} V_j(K|_G) \, dG ,$$
$j=0, \dots, k$, where $\cG_k^d$ is the Grassmann manifold of the $k$-dimensional subspaces of $\R^d$, integration is with respect to the Haar probability measure on $\cG^d_k$, and
$$  
 c_{d,k,j}  
 = \frac{k!  (d-j)!\kappa_{d-j} \kappa_k }{  d! (k-j)! \kappa_d \kappa_{k-j} } 
. $$
Because of \eqref{eq:dH<tau}, $ \l K \subset (\l K)_L + \tau(K) B^d$, and this also holds for all projections onto $k$-dimensional subspaces. Hence inequality \eqref{eq:dH<tau}
implies
$$
V_j(\l K|_G) -V_j( (\l K)_L|_G) \leq
\tau(K) \, 2V_{j-1} (\l K|_G),
$$
and Kubotas formula yields an upper bound for the intrinsic volumes.

\begin{corollary}
Let $K$ be a convex body. Then there is a constant $\tau(K)$, such that for $\l$ sufficiently large
$$ V_j(\l K)-\E (V_j((\l K)_L)) \leq \frac{2 c_{d,j-1,j}} {c_{d,j,j}} \tau(K) V_{j-1}(\l K)$$
for all $j \in \{1, \dots, d\}$.
\end{corollary}

\bigskip 
For a general lower bound on the mean width we need the following Lemma. It is a dual version of Blaschke's rolling theorem, and closely related to results of McMullen \cite{McMullen1974} and Sch\"utt and Werner \cite{SchuettWerner1990} for balls rolling inside a convex body. The dual version could be deduced from these results using a duality argument, and is stated explicitly in a paper by B\"or\"oczky, Fodor and Hug \cite{BoerFodorHug2013}.
\begin{lemma}[\cite{BoerFodorHug2013}, Lemma 5.2 ]\label{le:rollingball}
Let $K \in \cK^d$ be a convex body.
There exists a measurable set $\Sigma \subset \S^{d-1}$ with positive spherical Lebesgue measure, and some $R>0$, all depending on $K$, such that for any $u \in \Sigma$ there is some $p \in \bd K$ such that
$$
K \subset p +R (B^d -u) .
$$
\end{lemma}

The next theorem states the lower bound from Theorem \ref{thm:genbounds}.
\begin{theorem}\label{thm:genlowerbound}
Assume $K \subset B^d$ is a convex body. Then there is a constant $\g_1(K)$, such that
$$
W(\l K) - \E (W((\l K)_L))
\geq  \g_1 (K) \l^{-\frac{d-1}{d+1}}
$$
for $\l \geq 1$.
\end{theorem}

We prepare the proof of this theorem by the following lemma.
\begin{lemma}\label{le:cap-ball}
Let $B(0,r)$ be a ball of radius $r$. Then
$$
c_1 r^{ \frac{d-1}2} t^{\frac{d+1}2} \leq
V_d(B(0,r)_{t,u })
\leq c_2 r^{ \frac{d-1}2} t^{\frac{d+1}2} \ .
$$
\end{lemma}

\begin{proof}
For $r=1$, the intersection of $B^d$ with a hyperplane of distance $1-t$ from the origin is a $(d-1)$-dimensional ball with radius
$$\sqrt{2t -t^2} \in [t,2t] . $$
The volume of the cap $B^d_{t,u}$ is bounded from above by the volume of a cylinder and from below by the volme of a cone whose base are the same $(d-1)$-dimensional ball mentioned above.
Hence
$$
\frac 1d \kappa_{d-1} t^{\frac{d+1}2}
 \leq
V_d(B^d_{t,u})
\leq
2^{\frac{d-1}2} \kappa_{d-1} t^{\frac{d+1}2}  .
$$
Because
$$
V_d(B(0,r)_{t,u}) =
r^d V_d(B^d_{t/r,u})
$$
this proves the lemma.
\end{proof}

\begin{proof} [Proof of Theorem \ref{thm:genlowerbound}.]

We substitute $t = \l^{-\frac{d-1}{d+1}} x$ and obtain
\begin{align}\label{eq:Wdiff-smooth}
\nonumber
\l^{\frac{d-1}{d+1}} &\left(W(\l K) - \E (W((\l K)_L)) \right)
\\ \nonumber & =
\frac{2}{\omega_d}\int\limits_{\S ^{d-1}}\int\limits^{\infty}_{0} \l^{\frac{d-1}{d+1}} \P ((\l K)_{t,u}\cap L=\emptyset) dtdu
\\ &=
\frac{2}{\omega_d}\int\limits_{\S ^{d-1}}\int\limits^{\infty}_{0}  \P \left((\l K)_{ \l^{-\frac{d-1}{d+1}} x,u}\cap L=\emptyset \right) dxdu
\end{align}
By Lemma \ref{le:rollingball} there exists a suitable set $\Sigma \subset \S^{d-1}$ with $\l_{d-1} (\Sigma) > 0$ and a radius $R>0$ such that
$$
K \subset x +R (B^d -u) .
$$
For $u \in \Sigma$, by Lemma \ref{le:lower-bound-P} and Lemma \ref{le:L_meets_cap}  we have the lower bound
\begin{align*}
\P \left((\l K)_{ \l^{-\frac{d-1}{d+1}} x,u}\cap L=\emptyset \right)
& \geq
1-V_d \left( (\l K)_{\l^{-\frac{d-1}{d+1}} x,u} \right)
\\ &\geq
1-V_d \left(B(0,\l R)_{\l^{-\frac{d-1}{d+1}} x,u}\right)
\end{align*}
where we used that $\l K$ is contained in a ball of radius $\l R$.
Because of Lemma \ref{le:cap-ball}, we have
\begin{align*}
\P \left((\l K)_{ \l^{-\frac{d-1}{d+1}} x,u}\cap L=\emptyset \right)
&\geq
1-c_2 R^{ \frac{d-1}2} x^{\frac{d+1}2}
 .
\end{align*}
This implies
\begin{align*}
\l^{\frac{d-1}{d+1}} \left(W(\l K) - \E (W((\l K)_L)) \right)
& \geq
\frac{2}{\omega_d}\int\limits_{\Sigma}\int\limits^{\infty}_{0}  \P \left((\l K)_{ \l^{-\frac{d-1}{d+1}} x,u}\cap L=\emptyset \right) dxdu
\\ & \geq
\frac{2 }{\omega_d}
\int\limits_{\Sigma} du\  \int\limits^{\infty}_{0}  \left(1-c_2 R^{ \frac{d-1}2} x^{\frac{d+1}2}\right)_+ \   dx =
\\  & = \g_1 (K)
\end{align*}
where the constant $\g_1 (K)$ depends on $\Sigma $ and $ R$ and thus on $K$.

\end{proof}

\section{Smooth Convex Bodies}\label{sec:smooth}

In this section we prove a precise version of the upper bound in Theorem \ref{thm:smooth} concerning smooth convex bodies.
Fix the dimension $d \geq 2$, and for $r >0$ denote by $\cK(r)$ the set of convex bodies where a ball of radius $r$ rolls inside $K$, i.e.
$K \in \cK^d$ and for all $p \in \bd K$ there exist a unit vector $u \in \S^{d-1} $ with
$$
p + r (B^d -u)  \subset K .
$$

\begin{theorem}\label{thm:cKD}
Assume $K \in \cK(r)$. Then there is a constant $\g_3$ depending on $r$, such that
$$
W(\l K) - \E (W((\l K)_L)) \leq
\g_3 \l^{-\frac{d-1}{d+1}}
$$
for $\l$ sufficiently large.
\end{theorem}

\begin{proof}
To prepare for the use of Lemma \ref{le:upper-bound-P} in the following, we assume that $\l \geq \l(r)$ where $\l(r)$ is chosen such that 
$$ V_d(B(0, \l(r)\, r) = 2 \nu .$$
As in the proof of Theorem \ref{thm:genlowerbound} we start with
\begin{align*}
\l^{\frac{d-1}{d+1}} &\left(W(\l K) - \E (W((\l K)_L)) \right)
\\ &=
\frac{2}{\omega_d}\int\limits_{\S ^{d-1}}\int\limits^{\infty}_{0}  \P \left((\l K)_{ \l^{-\frac{d-1}{d+1}} x,u}\cap L=\emptyset \right) dxdu  .
\end{align*}

We first use that each boundary point of $\l K$ is touched from inside by a ball of radius $\l r$, and then Lemma \ref{le:upper-bound-P},
\begin{align}
\nonumber \l^{\frac{d-1}{d+1}} &\left(W(\l K) - \E (W((\l K)_L)) \right)
\\ \nonumber & =
\frac{2}{\omega_d}\int\limits_{\S ^{d-1}}\int\limits^{\infty}_{0}  \P \left(B(0, \l r)_{ \l^{-\frac{d-1}{d+1}} x,u}\cap L=\emptyset \right) dxdu
\\ \label{eq:I1} &\leq
2 \int\limits^{x_1}_{0} dx
+
2 \int\limits^{\infty}_{x_1} \frac c{V_d\left(B(0,  \l r)_{\l^{-\frac{d-1}{d+1}} x,u}\right)} dx .
\end{align}
Here $x_1=x_1(r)$ is chosen such that
$$V_d\left(B(0, \l r)_{\l^{-\frac{d-1}{d+1}} x_1,u}\right) = \nu,
$$
which by Lemma \ref{le:cap-ball} implies
$$
\left(\frac{\nu}{c_2}\right)^{\frac 2{d+1}} r^{-\frac{d-1}{d+1}}
\leq
x_1
\leq
\left( \frac{\nu}{c_1}\right)^{\frac 2{d+1}} r^{-\frac{d-1}{d+1}}
.
$$
The first integral in \eqref{eq:I1} is bounded by $2 x_1 $.

For the second integral in \eqref{eq:I1} we use Lemma \ref{le:cap-ball} to obtain
\begin{align*}
2 \int\limits^{\infty}_{x_1} \frac c{V_d\left(B(0, \l r)_{\l^{-\frac{d-1}{d+1}} x,u}\right)} dx
& \leq
\frac {2 c }{c_1}  r^{ -\frac{d-1}2}\int\limits^{\infty}_{x_1} x^{-\frac{d+1}2}   dx
\\ &=
\frac {4 c }{c_1 (d-1)}   r^{ -\frac{d-1}2}x_1^{-\frac{d-1}2} .
\end{align*}

Therefore,
$$ \l^{\frac{d-1}{d+1}} \left( W(\l K)-\E (W((\l K)_L)) \right) \leq \g_3 $$
where $\g_3$ depends on $r$. 
\end{proof}
It would be helpful, if for smooth $K \in \cK^d$ we have the convergence
$$
\P \left( (\l K)_{ \l^{-\frac{d-1}{d+1}} x,u} \cap L=\emptyset \right)
 \to
 f_K(x,u)
$$
as $\l \to \infty$, with some measurable function $f_K(x,u)$. Yet we have not been able to prove that.

\section{Polytopes}\label{sec:polytopes}
The preceding section shows that the lower bound in Theorem \ref{thm:genbounds} cannot be improved in general since it is - up to constants - sharp for smooth convex bodies. In this section we prove that also the upper bound is optimal - up to constants.

\begin{theorem}
Let $P \in \cP^d $ be a $d$-dimensional polytope with nonempty interior. Then there is a constant $\g(P)>0$ such that
$$ \lim\limits_{\l \to \infty} W(\l P)-\E (W((\l P)_L)) = \g(P) . $$
\end{theorem}

\begin{proof}

The polytope $P$ is the convex hull of its vertices $v \in \cF_0(P)$, and for each vertex $v $ we denote by $N(v) \subset \S^{d-1}$ the (relatively open) normal cone at $v$, i.e. the set of all unit vectors orthogonal to a supporting hyperplane to $P$ touching $P$ at $v$.
$$
N(v)=
\{ u \in \S^{d-1} \colon H_K(u)  \cap P =v \} .
$$
Since $P$ is a polytope, the set of unit vectors in 
$$\S^{d-1} \backslash \left(\bigcup_{v \in \cF_0(P) } N(v)\right) $$ 
is a null set with respect to spherical Lebesgue measure.
Thus Lemma \ref{le:EW=intP} gives
\begin{eqnarray*}
\lim_{\l \to \infty} \lefteqn{ \Big( W(\l P)-\E\, W((\l P)_L)\Big) = }
& &
\\ & = &
\sum_{v \in \cF_0(P) } \lim_{\l \to \infty} \int\limits_{N(v)} \int\limits^{\infty}_{0} \P ((\l P)_{t,u}\cap L=\emptyset)\, dtdu
\\ & = &
\sum_{v \in \cF_0(P) } \lim_{\l \to \infty} \int\limits_{N(v)} \int\limits_0^{\tau(P)} \P ((\l P)_{t,u}\cap L=\emptyset)\, dtdu
\end{eqnarray*}
where in the second line we used Lemma \ref{le:L_meets_cap}. Because the probability is bounded by $1$, Lebesgues dominated convergence theorem can be applied yielding
\begin{equation*}
\lim_{\l \to \infty} W(\l P)-\E\, W((\l P)_L)
 =
\sum_{v \in \cF_0(P) } \int\limits_{N(v)} \int\limits^{\tau (P)}_{0} \lim_{\l \to \infty} \P ((\l P)_{t,u}\cap L=\emptyset)\, dtdu .
\end{equation*}
By the translation invariance of the measure $dL$ on $\cL$,
$$
\P ((\l P)_{t,u}\cap L=\emptyset) =
\P ((\l (P-v))_{t,u}\cap L=\emptyset) $$
and the set $\l (P-v)$ converges to an infinite cone $C_v$ with apex in the origin, as $\l \to \infty$.
Thus for $u \in N(v)$ and $t$ fixed we have
$$
\lim_{\l \to \infty} \P ((\l P)_{t,u}\cap L=\emptyset) =
\P (C_v)_{t,u}\cap L=\emptyset) ,
$$
and the mean width difference converges to
\begin{eqnarray*}
\lim_{\l \to \infty} W(\l P)-\E\, W((\l P)_L)
&=&
\sum_{v \in \cF_0(P) } \int\limits_{N(v)} \int\limits^{\tau (P)}_{0} \P ((C_{v})_{t,u}\cap L=\emptyset)\, dtdu
\\ &=& \g(P) .
\end{eqnarray*}
We need some argument that the probability $\P ((C_{v})_{t,u}\cap L=\emptyset)$ is not vanishing.
By Lemma~\ref{le:lower-bound-P},
$$
\P ((C_{v})_{t,u}\cap L=\emptyset)
\geq 1- V_d((C_{v})_{t,u}).
$$
Observe that $(C_v)_{t,u}$ is a pyramid with height $t$, and thus the volume tends to $0$ for $t \to 0$. Therefore the probability is bounded from below by a function which is strictly positive in a neighborhood of $t=0$, hence $c(P)$ must be positive.

\end{proof}


\end{document}